\documentclass[a4paper,11pt]{amsart}

\usepackage{url}
\usepackage{amssymb}
\usepackage{color}
\usepackage[normalem]{ulem}
\usepackage[utf8]{inputenc}

\usepackage{cite}

\newcommand{\beaa}{\begin{eqnarray*}}

\newcommand{\qedskip}{$\qed$ \medskip}
\newcommand{\bean}{\begin{eqnarray}\nonumber}

\DeclareFontFamily{OT1}{rsfs}{}
\DeclareFontShape{OT1}{rsfs}{m}{n}{ <-7> rsfs5 <7-10> rsfs7 <10->
rsfs10}{}
\DeclareMathAlphabet{\mathscr}{OT1}{rsfs}{m}{n}
\newcommand{\noredA}{A}

\newcommand{\mcM}{\mathscr M}

\newcommand{\eq}[1]{\eqref{#1}}

\newcommand{\bel}[1]{\begin{equation}\label{#1}}
\newcommand{\beal}[1]{\begin{eqnarray}\label{#1}}
\newcommand{\beadl}[1]{\begin{deqarr}\label{#1}}
\newcommand{\eeadl}[1]{\arrlabel{#1}\end{deqarr}}
\newcommand{\eeal}[1]{\label{#1}\end{eqnarray}}
\newcommand{\eead}[1]{\end{deqarr}}
\newcommand{\eea}{\end{eqnarray}}
\newcommand{\eeaa}{\end{eqnarray*}}

\newcommand{\be}{\begin{equation}}
\newcommand{\ee}{\end{equation}}

\DeclareFontFamily{OT1}{rsfs}{}
\DeclareFontShape{OT1}{rsfs}{m}{n}{ <-7> rsfs5 <7-10> rsfs7 <10->
rsfs10}{} \DeclareMathAlphabet{\mycal}{OT1}{rsfs}{m}{n}

\newcounter{mnotecount}[section]

\newcommand{\N}{{\Bbb N}}

\newcommand{\rmnote}[1]{}%{\mnote{#1}}

\newcommand{\Ric}{\operatorname{Ric}}

\def\mysavedown#1{\edef\mysubs{\mysubs#1}}
\def\mysaveup#1{\edef\mysups{\mysups#1}}
\def\mydown#1{{\mytensor}_{\vphantom{\mysubs}#1}}
\def\myup#1{{\mytensor}^{\vphantom{\mysups}#1}}
\def\tensor#1#2{
  #1
  \def\mytensor{\vphantom{#1}}
  \def\mysubs{\relax}
  \def\mysups{\relax}
  \let\down=\mysavedown
  \let\up=\mysaveup
  #2
  \let\down=\mydown
  \let\up=\myup
  #2
  }

\newcommand{\Hess}{\operatorname{Hess}}

\newcommand{\Tr}{\operatorname{Tr}}
\newcommand{\grav}{\operatorname{grav}}

\newcommand{\R}{\mathbb R}

\newcommand{\mbbS}{\mathbb S}

\renewcommand{\setminus}{\smallsetminus}

\renewcommand{\to}{\rightarrow}
\renewcommand{\div}{\operatorname{div}}
\renewcommand{\epsilon}{\varepsilon}
\renewcommand{\hat}{\widehat}

\def\crn#1#2{{\vcenter{\vbox{
        \hbox{\kern#2pt \vrule width.#2pt height#1pt
           }
          \hrule height.#2pt}}}}

\newcommand{\griem}{\mathfrak G}
\newcommand{\glorentz}{\mathcal G}
\newcommand{\Ogr}{{O_{\mathring g}}}
\newcommand{\Ogrbar}{{O_{\rho^2 \mathring g}}}

\renewcommand{\hbar}{{\overline h}}

\newcommand{\pre}[2]{{{\vphantom{#2}}^{#1}}\kern-.2ex{#2}}

\theoremstyle{plain}
\newtheorem{theorem}{\sc Theorem}[section]
\newtheorem{lemma}[theorem] {\sc Lemma}
\newtheorem{proposition}[theorem]{\sc Proposition}

\newtheorem{prop}[theorem]{\sc Proposition}
\newtheorem{corollary}[theorem] {\sc Corollary}

\theoremstyle{definition}
\newtheorem{remark}[theorem]{\sc  Remark\rm}

\numberwithin{equation}{section}

\date{December 8, 2016.}

\renewcommand{\glorentz}{{ {\mathbf g}}}

\renewcommand{\griem}{{ {\mathfrak g}}}

\begin{document}

\title[Static  spacetimes with negative $\Lambda$] {{Non-singular spacetimes with a negative cosmological constant: II.  Static solutions of the Einstein-Maxwell equations}}

\thanks{Preprint UWThPh-2016-26}

\author[P.T. Chru\'sciel]{Piotr T.~Chru\'sciel}

\address{Piotr
T.~Chru\'sciel, Erwin Schr\"odinger Institute and Faculty of Physics, University of Vienna, Boltzmanngasse 5, A1090 Wien, Austria}
\email{piotr.chrusciel@lmpt.univ-tours.fr} \urladdr{
http://www.phys.univ-tours.fr$/\sim${piotr}}

\author[E.  Delay]{Erwann Delay}

\address{Erwann Delay, Avignon Universit\'e, Laboratoire de Math\'ematiques d’Avignon (EA 2151)
F-84018 Avignon} \email{Erwann.Delay@univ-avignon.fr}
\urladdr{http://www.math.univ-avignon.fr}

\begin{abstract}
We construct infinite-dimensional families of non-singular static
space times, solutions of the vacuum Einstein-Maxwell equations with a negative cosmological constant. The families include  an infinite-dimensional family of solutions with the usual AdS conformal structure at conformal infinity.

%{\bf keywords:} stationary solutions of Einstein-Maxwell equations with negative cosmological constant

\end{abstract}

\maketitle

\tableofcontents
\section{Introduction}\label{section:intro}

 Stationary and static solutions play a fundamental role in any theory.  Stable such solutions provide families
of possible end states of the problem at hand. Unstable ones display features of the solutions that are likely not to be encountered at late times of evolution.

Lichnerowicz's theorem asserts that the only vacuum stationary well-behaved solution with $\Lambda=0$ is Minkowski space-time. It therefore came as a surprise that there exist many stationary solutions of the vacuum Einstein equations with negative cosmological constant which have a smooth conformal structure~\cite{ACD,ACD2,ChDelayStationary}. In particular
there exist many globally well behaved stationary vacuum space-times with negative cosmological constant which have no symmetries whatsoever.

Now, Lichnerowicz's theorem generalises to static Einstein-Maxwell equations with $\Lambda=0$~\cite{Heusler:book}: there is again only one such regular solution, namely Minkowski space-time. It is likewise of interest to enquire what happens when $\Lambda <0$. In this work we show that, similarly, no Lichnerowicz-type theorem exists in this case:
we prove existence of infinite-dimensional families of static metrics satisfying the Einstein-Maxwell equations with a negative cosmological constant
and  admitting
a smooth conformal completion at infinity.

More precisely, we show that a
large class of such fields can be constructed by prescribing the
conformal class of a static Lorentzian metric and the asymptotic behaviour of the electric field on the conformal
boundary $\partial \mcM$, provided that the boundary class is
sufficiently close to, e.g., that of anti-de Sitter spacetime and  {the freely prescribable leading coefficient in an asymptotic expansion of the  electric potential is sufficiently small.}
This complements our previous work on vacuum metrics in~\cite{ChDelayStationary}, which we think of as paper I in this series.

The key new feature of the current result, as compared to~\cite{ACD,ACD2,ChDelayStationary}, is that we can drive the solution with the electric field  maintaining, if desired, a conformally flat structure at the conformal boundary at infinity.

 Once this work was completed we have noticed~\cite{BBSLMR}, where a class of numerical solutions of the Einstein-Maxwell equations with $\Lambda<0$  is presented. Our work justifies rigorously the existence of weak-matter-fields configurations within the family constructed numerically in~\cite{BBSLMR}, and provides many other static non-vacuum solutions near anti-de Sitter.

We thus seek to construct Lorentzian metrics ${}\glorentz$, solutions of Einstein-Maxwell equations with  {a negative cosmological constant $
\Lambda$, in any
space-dimension $n\geq 3$. {More precisely,
we consider the following field equations for a metric
$$
 \glorentz=g_{\mu\nu}dx^\mu dx^ \nu%\mathfrak G={\mathcal G}=
$$
and a two-form field $F$ in space-time dimension $n+1$, $n\ge 3$
generalising the Einstein-Maxwell equations in dimension $3+1$ as
\bel{11X16.1}
\Ric(\glorentz)-\frac{ \Tr \Ric(\glorentz)}2 {}\glorentz+\Lambda{}\glorentz=F\circ F-\frac{1}{4}\langle F,F\rangle_{\glorentz}\glorentz
 \,,
\ee
where $\Tr$ denotes a trace with the relevant metric (which should be obvious from the context), $(F\circ F)_{\alpha \beta} :=   g^{\mu\nu} F_{\alpha \mu} F_{\beta\nu}$, and
$$
 \langle F,F\rangle  \equiv  g^{\alpha \beta}  g^{\mu\nu} F_{\alpha \mu} F_{\beta\nu}=: |F|^2
 \,.
$$
%.
This is complemented with the Maxwell equations
\bel{11X16.2}
\mbox{div}_{\glorentz}F=0 =dF
 \,.
\ee
We further assume existence of  a hypersurface-orthogonal globally timelike} Killing vector
$X=\partial/\partial t$. In adapted coordinates the space-time metric can
be written as
\beal{gme1} &{}\glorentz = -V^2dt^2 + \underbrace{g_{ij}dx^i dx^j}_{=g}\,,
 &
\\
 &
\partial_t V  = \partial_t g=0\,.
\eeal{gme2}
{We further assume that the} Maxwell field takes the form
\beal{gme1+} &  F=d(Udt)\,, \quad \partial_t U  =0
 \,.
&
\eea

Our main result reads as follows (see Section~\ref{sec:def} below for the definition of
non-degeneracy; the function $\rho$ in \eq{result} is a coordinate
near the conformal boundary at infinity $\partial M$ that vanishes at $\partial M$):

\begin{theorem}\label{maintheorem}
 Let $n=\dim M\ge 3$, $k\in\N\setminus\{0\}$, $\alpha\in(0,1)$, and
 consider an Einstein metric $\mathring{\glorentz}$ as in
 \eq{gme1}-\eq{gme2} with strictly positive $V=\mathring V$,
 $g=\mathring g$, such that the associated Riemannian
 metric $\mathring{\griem}={\mathring V}^2d\varphi^2+\mathring g$ on
 $\mbbS^1\times M$ is $C^2$-compactifiable and non-degenerate, with
 smooth conformal infinity.  For every smooth $\hat{U}$,
 sufficiently close to zero in $C^{k+2,\alpha}(\partial M)$,
 there exists a unique, modulo diffeomorphisms which are the
 identity at the boundary, solution of the static Einstein-Maxwell equations of the form
 \eq{gme1}-\eq{gme1+} such that, in local coordinates near the
 conformal boundary $\partial M$,
 \beal{20XI16.1}
  & V-\mathring
 V=O(\rho)\,,\quad
 g_{ij}-\mathring g_{ij} =O(1)
 \,,
 &
\\
 &  \phantom{x} U=\hat U +\left\{
                               \begin{array}{ll}
                                 O(\rho), & \hbox{$n=3$;} \\
                                 O(\rho^2 \ln \rho), & \hbox{$n=4$;}\\
%                                 O(\rho^2), & \hbox{$n=4$ and  $[\rho^2\tilde g]_{S^1\times\partial M}$ is Ricci-flat.}\\
                                 O(\rho^2), & \hbox{$n\ge 5  $.}
                               \end{array}
                             \right.
 &
% \nonumber
\eeal{result}
\end{theorem}

\begin{remark}
  We have emphasised  the freedom to choose the leading-order behaviour $\hat U$ of $U$. As such,   the leading-order behaviour of both $\mathring V$ and $\mathring g$ is, similarly, freely prescribable near a non-degenerate solution~\cite{ACD2,ChDelayStationary}. Here one can, if desired, proceed in two steps: first, find the static solution with new fields $\mathring V$ and $\mathring g$ near a non-degenerate solution; {small such perturbations will preserve non-degeneracy. One can then use Theorem~\ref{maintheorem} at the perturbed static solution} to obtain a solution with the desired small $\hat U$.
  \qed
\end{remark}

The $(n+1)$-dimensional anti-de Sitter metric is non-degenerate in the
sense above (see eg. \cite{ACD2} appendix D),
so Theorem~\ref{maintheorem} provides in particular an
infinite dimensional family of solutions near that metric.  We note existence of further large classes of Einstein
metrics satisfying the non-degeneracy condition~\cite{Lee:fredholm,ACD2,mand1,mand2}.

The requirement of strict positivity of $\mathring V$ excludes black
hole solutions,
we are hoping to return to the construction of black-hole solutions in a near future.   In fact, this work was prompted by~\cite{HerdeiroRadu}, where static Einstein-Maxwell black holes solutions driven by the asymptotics of the electric field have been constructed numerically.

{Similarly to~\cite{HerdeiroRadu}, for generic $\widehat U$ the resulting space-time metric will have no isometries other than time-translations.}

{When the free boundary data $\hat U$ are smooth and when $\mathring \griem$ is conformally smooth,} the solutions constructed here will have a polyhomogeneous expansion at the conformal boundary at infinity. The proof of this is an immediate repetition of the argument presented in~\cite[Section~7]{ChDelayStationary}, where the reader can also find the definition of polyhomogeneity. The decay rates in \eq{20XI16.1}-\eq{result} have to be compared with the
{local-coordinates} leading-order behavior $\rho^{-2}$ both for ${\mathring V}^2$ and
${\mathring g}_{ij}$. A {more} precise version of \eq{20XI16.1}-\eq{result} in terms of
weighted function spaces {(we follow the notation in~\cite{Lee:fredholm,ChDelayStationary})} reads,  {in all dimensions},
\beal{result2}
 & (V-\mathring V)\in
C_{1}^{k+2,\alpha}( M)\,,\,\,(g-\mathring g)\in
C_{2}^{k+2,\alpha}(M,{\mathcal S}_2)\,,
 & \\ &
 \quad
U-\hat U \in C_{1}^{k+2,\alpha}( M)\,,
 &
\eeal{result3}
and the norms of the differences above are small in those spaces. To this one can add, in dimensions $n\ge 4$,
\beal{result2+}
  &
 \quad
U-\hat U - U_{\ln}\, \rho^2 \ln \rho  \in C_{2}^{k+2,\alpha}( M)
\,,
 &
\eea
with the function $U_{\ln} \in C^\infty( \mbbS^1\times \overline M) $ given by \eq{25XI16.21} below when $n=4$, and $  U_{\ln} \equiv 0$ in dimension $n\ge 5$.  Interestingly enough, in dimension $n=4$ the function $U_{\ln{}}$ vanishes only when $d U \equiv 0$,
in which case  the space-time is vacuum, see Remark~\ref{Ulncst} below.

 The proof of Theorem~\ref{maintheorem} can be found at the end of Section~\ref{s13XI16.1}.
{It follows closely~\cite{ChDelayStationary},} and proceeds through an implicit-function
argument, {with the key isomorphism properties of the
associated linearised operator borrowed from~\cite{Lee:fredholm,ChDelayStationary}, and presented in Section~\ref{sec:iso}.}

As already noted,  a constant $\hat U$ implies vacuum, by uniqueness of solutions. Hence the energy density of the Maxwell field behaves as $\rho^{4}$ for small $\rho$  if not identically zero (cf.~\eq{16XI16.12} below),  which leads to finite-total-energy configurations in space-dimensions $n=3$ and $4$, but infinite matter energy in higher dimensions.

\section{Definitions, notations and conventions}\label{sec:def}

Our definitions and conventions are identical to those in~\cite[Section~2]{ChDelayStationary}.
Here we simply recall that the  Lichnerowicz Laplacian acting on a symmetric
two-tensor field is defined as~\cite[\S~1.143]{besse:einstein}
$$
\Delta_Lh_{ij}=-\nabla^k\nabla_kh_{ij}+R_{ik}h^k{_j}+R_{jk}h^k{_i}-2R_{ikjl}h^{kl}\,.
$$
The operator $\Delta_L+2n$ arises naturally when  linearising the equation
\bel{EE}
\Ric(\griem)=-n \griem\,,
\ee
where $\Ric(\griem)$ is the Ricci curvature of $\griem$, at a solution.
We will say that a metric is {\it non-degenerate} if $\Delta_L+2n$ has no
$L^2$-kernel.

\section{Isomorphism theorems}\label{sec:iso}

In this section we recall two isomorphism results from~\cite[Section~3]{ChDelayStationary} and prove an isomorphism theorem on scalar fields, as needed in the remainder of this work. The reader might also want to consult the introductory comments in~\cite[Section~3]{ChDelayStationary}, which remain valid in the current context.

%
%Some of the isomorphism theorems we will use are consequences of
% Lee's theorems~\cite{Lee:fredholm}, it is therefore convenient to
% follow his notation for the weighted H\"older spaces
% $C^{k,\alpha}_\delta$. As described in the second paragraph before
%Proposition B of~\cite{Lee:fredholm}, a tensor in this space
% corresponds to $\rho^\delta$ times a tensor in the usual
% $C^{k,\alpha}$ space as defined  using the norm of the conformally
% compact metric. This implies that, in local coordinates near the
% conformal boundary, a function in $C^{k,\alpha}_\delta$ is
% $O(\rho^\delta)$, a one-form in $C^{k,\alpha}_\delta$ has components
% which are $O(\rho^{\delta-1})$, and a covariant two-tensor in
% $C^{k,\alpha}_\delta$ has components which are $O(\rho^{\delta-2})$.
%
%As such, under the regularity
%conditions on the metric in our definition of asymptotically
%hyperbolic metric, the  isomorphism theorems of~\cite{Lee:fredholm} in
%weighted $C^{k,\alpha}$ spaces apparently apply only to low values
%of $k$. However, under our hypotheses, one can use those theorems for
%$k=2$, and use scaling estimates to obtain the conclusion for any
%value of $k$.

\subsection{An isomorphism on two-tensors}

We will need~\cite[Corollary~3.2]{ChDelayStationary}:

\begin{corollary}
 \label{isotordu}
Let $(\mbbS^1 \times M,V^2 d\varphi^2 + g)$ be an asymptotically hyperbolic non-degenerate  Riemannian manifold  with $\partial_\varphi V = 0= \partial _\varphi g $. Consider the map
$$
  (W,h)\mapsto (l(W,h),L(W,h))\,,
$$
where
\begin{eqnarray}
 \label{16XO16/2}
 \lefteqn{ l(W,h)
  =
  V\displaystyle{\big[
 (\nabla^*\nabla+2n+V^{-1}\nabla^*\nabla V+V^{-2}|dV|^2)
 W }
 }
 &&
\\
 &&
 +V^{-1}\nabla_jV\nabla^jW
 -V^{-1}\nabla^jV\nabla^kVh_{kj}+\langle
 \Hess_gV,h\rangle_g \big]
 \,,
  \nonumber
\end{eqnarray}
 and
\begin{eqnarray}
 \label{16XO16/1}
L_{ij}(W,h)&=&\frac{1}{2}\Delta_L h_{ij}+nh_{ij}-\frac{1}{2}V^{-1}\nabla^kV\nabla_kh_{ij}
  \\
&&+\frac{1}{2}V^{-2}(\nabla_iV\nabla^kVh_{kj}+\nabla_jV\nabla^kVh_{ki})
  \nonumber
  \\
&&-\frac{1}{2}V^{-1}(\nabla_i\nabla^kVh_{kj}+\nabla_j\nabla^kVh_{ki})
  \nonumber
\\
&&+2V^{-2}W (\Hess_g V)_{ij} -2V^{-3}\nabla_iV\nabla_jV W
 \,.
  \nonumber
\end{eqnarray}
Then $(l,L)$ is an isomorphism from
$C^{k+2,\alpha}_{\delta-1}(M)\times C^{k+2,\alpha}_{\delta}(
M,{\mathcal S}_2)$ to $C^{k,\alpha}_{\delta-2}(M)\times
C^{k,\alpha}_{\delta}( M,{\mathcal  S}_2)$ when $\delta\in(0,n)$.
\end{corollary}

\subsection{An  isomorphism on one-forms}
 \label{ss7XII16.1}

We will also need the following sharper version of~\cite[Corollary~3.5]{ChDelayStationary}, where we note that the isomorphism  given in Theorem 3.4  and Corollary 3.5 of~\cite{ChDelayStationary} are in fact valid for a larger interval of weights, as can be established by a more careful inspection of the arguments presented in~\cite{ChDelayStationary}:

\begin{corollary}
 \label{cor:isoform}
Let $k\in \N$, $\alpha \in (0,1)$.
Under the  hypotheses of Corollary~\ref{isotordu}, suppose moreover that the Ricci tensor of $V^2 d\varphi^2 + g$ is negative. Consider the operator
$$
\Omega_i \mapsto
{B}(\Omega)_i+R_{ij}\Omega^j-V^{-1}\nabla_i\nabla^jV\Omega_j=:{\mathcal
B}(\Omega)_i\,,
$$
where
$$
{
B}(\Omega)_i:=\nabla^k\nabla_k\Omega_i+V^{-1}\nabla^kV\nabla_k\Omega_i
 -V^{-2}\nabla_iV\nabla^kV\Omega_k\,.
 $$
Then $\mathcal B$ is an isomorphism from $C^{k+2,\alpha}_{\delta}(
M,{\mathcal T}_1)$ to $C^{k,\alpha}_{\delta}(M,{\mathcal T}_1)$
when
$$\left|\delta-\frac{n}{2}\right|<\frac{n+2}2.$$
%$|\delta-\frac{n}{2}|<\sqrt{\frac{n^2}{4}+1}$.
\end{corollary}

\subsection{An isomorphism on functions}

If we assume that $V^2d\varphi^2+g$ is a static asymptotically
hyperbolic metric on $\mbbS^1\times M$, then it is easy to check that
\bel{30XI16.1}
 \mbox{$V^{-2}|dV|^2\to 1$ and $V^{-1}\nabla^i\nabla_iV\to n$}
\ee
as the conformal boundary is approached.  We will need an isomorphism property for the following
operator acting on functions {with $s=-1$;  in~\cite{ChDelayStationary} the result has already been established with $s=-3$ and $s=3$, so for future reference it appears useful to consider all values of $s$}:
$$
\sigma\mapsto {\mathcal
T}_s\sigma:=V^{-s}\nabla^i(V^{s}\nabla_i\sigma)=
\nabla^i\nabla_i\sigma+sV^{-1}\nabla^iV\nabla_i\sigma\,.
$$

\begin{theorem}\label{isofunction}
{
Let $(M,g)$ be an $n$-dimensional Riemannian manifold with an asymptotically hyperbolic metric with $V>0$  and assume that \eq{30XI16.1} holds.
}
%
%Let $(V,g)$ be close in $C_{-1}^{k+2,\alpha}(M)\times
%C_{0}^{k+2,\alpha}(M,{\mathcal S}_2)$ to an asymptotically hyperbolic
%static metric.
Let $s\neq 1-n$ and suppose that
 $$
 \frac{s+n-1-|s+n-1|}2<\delta<\frac{s+n-1+|s+n-1|}2
 \,.
$$
If $s\geq -\frac{n-1}{2}$
then $\mathcal T_s$ is an isomorphism from
$C^{k+2,\alpha}_{\delta}( M)$ to $C^{k,\alpha}_{\delta}(M)$.
If $s<-\frac{n-1}{2}$,
then ${\mathcal T}_s$ is an
isomorphism from $C^{k+2,\alpha}_{\delta}( M)/\R$ to
\bel{intcond}\Big\{\sigma\in C^{k,\alpha}_{\delta}(M): \int_M V^{s}\sigma=0\Big\}\,.\ee
\end{theorem}

\begin{proof}
When $s+n-1>0$, we can  use
\cite[Theorem~7.2.1 (ii) and Remark (i), p.~77]{AndChDiss}
to conclude.  As we want the result for any $s\neq 1-n$, for the remaining cases we  appeal to the results of
Lee~\cite{Lee:fredholm}.  For this we need a formally
self-adjoint operator, so we set $\sigma=V^{-\frac{s}{2}}f$, thus
\bel{ZcalZ}
 {\mathcal
T}_s\sigma=V^{{-}\frac{s}{2}}\left[\nabla^i\nabla_if-\frac s2\left((\frac s2-1)
V^{-2}|dV|^2+V^{-1}\nabla^i\nabla_iV\right)f\right]=:
 V^{{-}\frac{s}{2}}T_sf\,.
\ee
{By assumption} $V^{-2}|dV|^2\to 1$ and
$V^{-1}\nabla^i\nabla_iV\to n$ at the conformal boundary,
leading to the following indicial exponents  {for $T_s$:}
$$
  \delta=\frac{n-1\pm|s+n-1|}2\,.
$$
We want to show that $T_s$ satisfies condition (1.4) of
\cite{Lee:fredholm},
\bel{Lec1.4}\|u\|_{L^2}\le C \|{
T}_su\|_{L^2}
 \,,
\ee
for smooth $u$ compactly supported in a
sufficiently small open set ${\mathcal U}\subset M$ such that
$\overline{\mathcal U}$ is a neighborhood of $\partial M$.  Let us recall  the following, well known result
(see eg.~\cite[Lemma~3.8]{ChDelayStationary}):

\begin{lemma}\label{Lnew}
On an asymptotically hyperbolic manifold $(M,g)$ with boundary
defining function $\rho$ we have, for all compactly supported $C^2$
functions,
$$
\int u\nabla^*\nabla u\geq
\left(\frac{n-1}{2}\right)^2\int(1+O(\rho)) u^2.
$$
\end{lemma}

Lemma~\ref{Lnew} combined with the hypothesis (\ref{30XI16.1}) shows that
$$
\|u\|_{L^2}\|T_su\|_{L^2}\ge -\int uT_su
 \geq
 \int\Big(\frac{(n-1)^2}{4}+\frac s2(\frac s2+n-1)+o(1) \Big)u^2
 \,
 ,
$$
from which it follows that $ T_s$ satisfies indeed \eq{Lec1.4}
with
$$
 C^{-1}=\frac{(s+n-1)^2}{4}\,.
$$

We recall that the critical fall-off for a function to belong to $L^2$ is $O(\rho^{\frac{n-1}2})$. This can be used to show that the
$L^2$-kernel of $T_s$
equals
$$
\ker T_s=V^{\frac{s}{2}}\R\cap L^2=
\left\{\begin{array}{l}\{0\} \mbox{ if }s\geq -\frac{n-1}2
 \,,
 \\
V^{\frac{s}{2}}\R \mbox{ if } s<-\frac{n-1}2
 \,.
  \end{array}\right.
$$
Indeed, assume that $f$ is in the $L^2$-kernel of $T_s$. By elliptic
regularity (see  \cite[Lemma 4.8]{Lee:fredholm} for instance)
 $f$ is $H^2$ on $M$.  {This implies in particular that no boundary terms arise in the following integration by parts:}
\begin{eqnarray*}
0 & = & - \int fT_s f =-\int    f V^{-\frac s2}
\nabla^i[V^{s}\nabla_i(V^{-\frac s 2}f)]\\
& = & \int    V^{s}|\nabla (V^{-\frac s2}f)|^2.
\end{eqnarray*}
 So $V^{-\frac s2}f$ is a
constant.  Using~\cite{Lee:fredholm}, Theorem C(c), we find:

\begin{itemize}
  \item
If $s\geq \frac{n-1}2$ then the kernel (and then the cokernel) of $T_s$ is trivial
so $T_s$ is an isomorphism from
$C^{k+2,\alpha}_{\delta-\frac s2}( M)$ to $C^{k,\alpha}_{\delta-\frac s2}( M)$.

 \item
If $s<-\frac{n-1}{2}$, then $T_s$ is an isomorphism from
$C^{k+2,\alpha}_{\delta-\frac s2}( M)/V^{s/2}\R$ to the orthogonal to the kernel:
$$
 \Big\{f\in C^{k,\alpha}_{\delta-\frac s2}(M): \int_M V^{s/2}f=0\Big\}\,.
$$
\end{itemize}

The conclusion follows for $\mathcal T_s$ if we recall that $\sigma=V^{-\frac s2}f$
is in $C^{k,\alpha}_{\delta}(M)$ iff $f\in C^{k,\alpha}_{\delta-\frac s2}(M)$.
\end{proof}

\section{The equations}

Rescaling the metric to achieve a convenient normalisation of the
cosmological constant,
\bel{16XI16.11}
\Lambda=-\frac{n(n-1)}2
 \,,
\ee
the vacuum Einstein-Maxwell equations for a metric
satisfying \eq{gme1}-\eq{gme2} read (see Appendix~\ref{s15XI16.1})
\begin{equation}\label{mainequation}
\left\{\begin{array}{l} V(\nabla^*\nabla V+n V)=-\frac {n-2}{n-1} |dU|_g^2
\,,\\
    \Ric(g)+n g-V^{-1}\Hess_gV=V^{-2}\left(-dU\otimes dU+\frac{1}{n-1}|dU|_g^2g\right)\,,
\\
\div_g (V^{-1} \nabla U)=0\,.
\end{array}\right.
\end{equation}

\subsection{The linearised equation}\label{sec:line}

{We use the symbol $\Tr$, or $\Tr_g$ when the metric could be ambiguous, to denote the trace.}
As in \cite{ChDelayStationary}  we set
$$
 \grav h=h-\frac{1}{2}\Tr_g hg,\,\,\, (\div
 h)_i=-\nabla^kh_{ik},\,\,\,
 (\div^*w)_{ij}=\frac{1}{2}(\nabla_iw_j+\nabla_jw_i)
% \,,
$$
(note the geometers' convention, in which we have a negative sign in the definition of
divergence).
We  consider the operator from the set of functions times
symmetric two tensor fields to itself, defined as
$$
\left(\begin{array}{l}V\\g\end{array}\right)\mapsto
\left(\begin{array}{l}V(\nabla^*\nabla
V+nV)\\\Ric(g)+ng-V^{-1}\Hess_g V\end{array}\right)\,.
$$
{(To avoid ambiguities, $\nabla^*\nabla \equiv -\nabla^i \nabla_i =: - \Delta$.)}
The two components of its linearisation at $(V,g)$ are
\begin{eqnarray*}
p(W,h)
 & = & V\Big[(\nabla^*\nabla+2n+V^{-1}\nabla^*\nabla V) W
+ \langle \Hess_gV,h\rangle_g
\\
 &&
 -\langle \div\grav h,dV\rangle_g\Big]\,,
\\
 P_{ij}(W,h)
 &=&\frac{1}{2}\Delta_Lh_{ij}+nh_{ij}+\frac{1}{2}V^{-1}\nabla^kV(\nabla_ih_{kj}+
\nabla_jh_{kj}-\nabla_kh_{ij})
\\
&&
 -(\div^*
\div\grav h)_{ij}+V^{-2}W(\Hess_gV)_{ij}-V^{-1}(\Hess_g W)_{ij}\,.
\end{eqnarray*}
 It turns out to be convenient to introduce the one-form
$$
w_j=V^{-1}\nabla^kVh_{kj}+\nabla^kh_{kj}-\frac{1}{2}\nabla_j(\Tr
h)-V^{-1}\nabla_jW-V^{-2}\nabla_jV W\,,
$$
which allows us  to rewrite  $P(W,h)$ as %follows:
\begin{eqnarray*}
P(W,h)
%% &=&\frac{1}{2}\Delta_L h+nh-\frac{1}{2}V^{-1}\nabla^kV\nabla_kh_{ij}\\
%% &&+\frac{1}{2}V^{-2}(\nabla_iV\nabla^kVh_{kj}+\nabla_jV\nabla^kVh_{ki})\\
%% &&-\frac{1}{2}V^{-1}(\nabla_i\nabla^kVh_{kj}+\nabla_j\nabla^kVh_{ki})
%% \\
%% &&+2V^{-2}W \Hess_g V -2V^{-3}W\nabla_iV\nabla_jV +\div^*w\\
&=&L(W,h)+\div^*w\,,
\end{eqnarray*}
where $L$ is as in Corollary~\ref{isotordu}.  Similarly, $p(W,h)$ can
be rewritten as
\begin{eqnarray*}
p(W,h)
%&=&
%V
%\displaystyle{\left[(\nabla^*\nabla+2n+V^{-1}\nabla^*\nabla V+V^{-2}|dV|^2)
%W+V^{-1}\nabla_jV\nabla^jW\right.}\\&&\left.-V^{-1}\nabla^jV\nabla^kVh_{kj}+\langle
%\Hess_gV,h\rangle_g
%+\langle w,dV\rangle_g\right]\\
&=&l(W,h)+V\langle w,dV\rangle_g\,,
\end{eqnarray*}
with $l$ given by \eq{16XO16/2}.

\subsection{The modified equation}
 \label{secjauge}

{In Section~\ref{s13XI16.1} we will use the implicit function theorem to construct our
solutions, using the observation of~\cite{ChBActa}, how to obtain a well-behaved equation by adding ``gauge
fixing terms".} We choose those
terms as in~\cite{ChDelayStationary}, appealing to harmonic maps for the vacuum Einstein
equations in one dimension higher.

Indeed, we will be solving the following  system of
equations
\begin{equation}\label{jaugeequation}
\left\{\begin{array}{l}
 V(\nabla^*\nabla V+n V +\langle  \Omega,dV\rangle)+\frac {n-2}{n-1} |dU|_g^2=0
\,,\\
    \Ric(g)+n g-V^{-1}\Hess_gV+\div^*\Omega\\
		\hspace{3cm}-V^{-2}\left(-dU\otimes dU+\frac{1}{n-1}|dU|_g^2g\right)=0\,,
\\
\div_g (V^{-1} \nabla U)=0\,,
\end{array}\right.
\end{equation}
where
\newcommand{\tGamma}{ {\widetilde\Gamma}}%
\newcommand{\zGamma}{\mathring\Gamma}%
\newcommand{\hGamma}{\hat\Gamma}%
\newcommand{\tilg}{ {\mathring g}}%
\newcommand{\wg}{\hat g}%
\newcommand{\znabla}{\mathring\nabla}%
\newcommand{\tilV}{ {\mathring V}}%
\newcommand{\zV}{\tilV }%
\begin{eqnarray}
 \label{Omegaeq}
%\nonumber
 -\Omega_j&\equiv&
 -\Omega(V,g,\tilV ,\tilg )_j
\\
 \nonumber
&:=&
\griem_{j\mu}\griem^{\alpha \beta} (\Gamma(\griem)^\mu_{\alpha\beta}- \Gamma(\mathring{\griem})^\mu_{\alpha\beta})
%\\
%\nonumber
%& = &
%{g}_{jk}{g}{}^{\ell m} (\Gamma^k_{\ell m}- \zGamma^k_{\ell m}) +
%V^{-2}g_{jk}(\zV\znabla^k \zV - V\nabla^k V)
\\
& = &
{g}{}^{\ell m} (\znabla_m {g}_{j\ell}- \frac 12 \znabla_jg_{\ell m}) + V^{-2}g_{jk}(\zV\znabla^k \zV - V\nabla^k V)
 \,,
 \nonumber
\end{eqnarray}
and
where $\mathring\nabla$-derivatives are relative to the metric $
  \mathring g
 $,
with Christoffel symbols
$\zGamma^\alpha_{\beta\gamma}$, Latin indices run from $0$ to
$n$, and $\griem:=V^2(dx^0)^2+g$ with Christoffel symbols
$\Gamma(\griem)^\alpha_{\beta\gamma}$, while the
$\Gamma(\mathring{\griem})^\alpha_{\beta\gamma}$'s are the Christoffel symbols of the
metric $\mathring{\griem}$.

The derivative of $\Omega$ with respect to $(V,g)$ at $(\zV,\tilg )$ is
$$
 D_{(V,g)}\Omega(\zV,\tilg,\zV,\tilg )(W,h)=-w,
$$
where $w$ is the one-form defined in
Section~\ref{sec:line} with $(V,g)$ replaced with $(\zV,\tilg )$.  Thus, the
linearisation of $(q,Q)$ at $(\zV,\tilg )$ is
$$ D(q,Q)(\zV,\tilg )=(l,L)\,,$$ where $(l,L)$ is the operator defined in
Section~\ref{sec:line} with $(V,g)$ replaced with $(\zV,\tilg )$.  We will
show that, under reasonable conditions, solutions of
  \eq{jaugeequation} are solutions of
\eq{mainequation}).

\begin{proposition}
 \label{prop:solmodgivesol}
If $(U,V,g)$ solves
\eq{jaugeequation} then $\Omega$ is in the kernel of ${\mathcal B}$.   If moreover
 $ \Omega\in C^{2,\alpha}_\delta$  for some $\delta>-1$,
then $\Omega \equiv 0$, so that   $(U,V,g)$ solves \eq{mainequation}.
\end{proposition}

\begin{proof}
{
A standard adaptation of the argument in~\cite{ChDelayStationary} gives the result; we sketch this for completeness. Set}
\begin{eqnarray*}
 &
 a:= -\frac{n-2}{n-1} V^{-2} |d U|^2
  \,,
&
 \\
&
\noredA
 := V^{-2}(-dU\otimes dU+\frac{1}{n-1}
|d U|^2 g) \,.
 &
\end{eqnarray*}
With this notation, the first two equations in \eq{jaugeequation} take the form
\begin{equation}
 \label{jaugeequationag}
\left\{\begin{array}{l} \nabla^*\nabla V+n V+\langle \Omega,dV\rangle=V a
\,,\\
    \Ric(g)+n g-V^{-1}\Hess_gV+\div^*\Omega=
\noredA \,,
 \end{array}\right.
\end{equation}
The tensor field
$E(g)=\grav_g \Ric(g)$,
has vanishing divergence, which provides
{the equation ${\mathcal B}(\Omega)=0$ for $\Omega$. Indeed, the operator ${\mathcal B}(\Omega)$ defined above has been constructed in~\cite{ChDelayStationary} so that it vanishes when the modified equation for the metric holds and when the Lorentzian  $(n+1)$-dimensional divergence of the $(n+1)$-dimensional energy-momentum tensor vanishes. The latter condition is satisfied when the $(n+1)$-dimensional matter field equations hold; in the current case, this coincides with the last equation in \eqref{jaugeequation}.

An alternative, $n$-dimensional argument proceeds as follows: The calculation following \cite[Equation~(4.6)]{ChDelayStationary} shows that for a solution to the modified equation,
the divergence of $E(g)$ equals
$$
 0  \equiv \div E(g)=\frac{1}{2}{\mathcal B}(\Omega) + \beta
 \,,
$$
where
\bel{21XI16.2}
\beta:=V^{-1}d(Va) - V^{-1}
 \noredA (\nabla V,.)+
 \div(\grav_g \noredA +\frac{a}{2}g)
 \,.
\ee
The vanishing of $\beta$ can be checked by a somewhat lengthy calculation using \eq{jaugeequation}.}

Either way, the Bianchi identity $\div E(g)=0$ shows that
$\Omega$ is in the kernel of ${\mathcal B}$. It follows  from
Corollary~\ref{cor:isoform} that the only solution of this equation
which decays {as described} is zero.
\end{proof}

For future reference we consider the static equations, modified as in \eq{jaugeequationag}, with a general energy-momentum tensor:
\begin{eqnarray}
 \label{15XI16.21+}
  &
  \displaystyle
  -V^ {-1} \nabla^i \nabla_i V  - \frac{ 2 \Lambda} {n-1} +  V^{-1} \nabla^i\Omega \nabla_i V=
\underbrace{ - T_{\alpha \beta} N^\alpha N^\beta - \frac{g^{\alpha\beta}T_{\alpha\beta}  } {n-1}}
    _{     =: a}
  \,,
  &
\\
  &
  \displaystyle
   \phantom{xxx}
 R_{ij}- V^{-1} \nabla_i\nabla_j V - \frac{2\Lambda}{n-1} g_{ij} +\frac 12 (\nabla_i \Omega_j + \nabla_j \Omega_i)=
  \underbrace{ T_{ij} - \frac {g^{\alpha\beta}T_{\alpha\beta} } {n-1} g_{ij}}_{=:\noredA _{ij}}
 \,,
 &
\eeal{14XI16.1+ }
where $N^\alpha\partial_\alpha$ is the unit timelike normal to the initial data surface
(compare \eqref{15XI16.21}-\eqref{14XI16.1}, Appendix~\ref{s15XI16.1}). We have:

\begin{prop}
  \label{p20XI16.1}
Let $T_{\mu\nu}$ be divergence-free with respect to the time-independent metric
$$
\glorentz = -V^2 dt^2 + g
$$
as a consequence of the matter field equations, with $\partial_t T_{\mu\nu} = T_{0i}=0$. Set
\bel{28XI16.21}
  a:=   - T_{\alpha \beta} N^\alpha N^\beta - \frac{g^{\alpha\beta}T_{\alpha\beta}  } {n-1}
  \,,
  \quad
 \noredA _{ij} :=  T_{ij} - \frac {g^{\alpha\beta}T_{\alpha\beta} } {n-1} g_{ij}
  \,.
\ee
Then
\bel{28XI16.1}
 \beta \equiv 0
 \,.
\ee
\end{prop}

\proof
Equation~\eq{21XI16.2} can be rewritten as
\beal{21XI16.2+}
 \phantom{xxx}
-\beta_j  & = &
 \nabla_i\big( \noredA ^i{}_j - \frac 12 (\noredA ^k{}_k  + {a}) \delta^i_j\big)   +  V^{-1}
 \big(\noredA _{ij} - a g_{ij})\nabla V ^i
\\
 \nonumber
 & = &
 \nabla_i\Big( {T^i{}_j
   - \frac 12 (\frac {2g^{\alpha\beta}T_{\alpha\beta} } {n-1} +\noredA ^k{}_k  + {a}) \delta^i_j}\Big)
\\
 &&     +  V^{-1}
 \Big(T_{ij} - (\frac { g^{\alpha\beta}T_{\alpha\beta} } {n-1}
   -  T_{\alpha \beta} N^\alpha N^\beta
    - \frac{g^{\alpha\beta}T_{\alpha\beta}  } {n-1}) g_{ij}\Big)\nabla V ^i
 \,.
 \nonumber
\eea
We have
\begin{eqnarray*}
 \lefteqn{
  \frac {2g^{\alpha\beta}T_{\alpha\beta} } {n-1} + \noredA ^k{}_k +a
  }
  &&
\\
 &&
  = \frac {2g^{\alpha\beta}T_{\alpha\beta} } {n-1} +
    g^{ij} T_{ij} - \frac {n } {n-1}g^{\alpha\beta}T_{\alpha\beta}
   -
     T_{\alpha \beta} N^\alpha N^\beta
     -
      \frac{g^{\alpha\beta}T_{\alpha\beta}  } {n-1} =0
  \,.
\end{eqnarray*}
The $(n+1)$-decomposition of the equation $ \nabla^\mu T_{\mu\nu}=0$,   where $\nabla_\mu$ is the space-time covariant derivative associated with $\glorentz$,  using the formulae for the Christoffel symbols in  Appendix~\ref{s15XI16.1} gives
\bel{20XI16.3}
 \nabla^i  T_{ij} +V T_{\alpha \beta} N ^ \alpha N^\beta\nabla_j V
 + V^{-1} \nabla^j V T_{ij} = 0
 \,,
\ee
where $\nabla^i$ is the covariant derivative operator of $g$. Comparing with \eq{21XI16.2+} gives the result.
\hfill\qedskip

\section{The construction}
 \label{s13XI16.1}

 We consider  an asymptotically hyperbolic Einstein
static metric $\mathring{\griem}={\mathring V}^2d\varphi^2+{\mathring g}$ on
$\mbbS^1\times M$. We apply Theorem~\ref{isofunction} with $s=-1$,  $g$
--- a Riemannian metric on $M$ close to ${\mathring g}$ in
$C_{0}^{k+2,\alpha}(M,{\mathcal S}_2)$, and $V$ --- a function on $M$
close to ${\mathring V}$ in $C_{-1}^{k+2,\alpha}(M)$. It is convenient to choose some
\bel{16XI16.23}
 \mbox{
$\delta\in(0,1)$ when $n=3$ and $\delta=1$ if $n>3$.}
\ee
We conclude that for any $\hat{U}\in C^{k+2,\alpha}(\partial M)$,
there
exists a unique  solution
$$ U=U(\hat{U},V,g)\in C_{0}^{k+2,\alpha}(M)$$ to
$$
\left\{\begin{array}{l} \nabla^*(V^{-1}\nabla U)=0,\\
U-\hat{U}\in C_{\delta}^{k+2,\alpha}(M)\,.
\end{array}\right.
$$
 Moreover, the map $(\hat{U},V,g)\mapsto
U-\hat{U}\in C_{\delta}^{k+2,\alpha}(M)$ is smooth.\\

We define a new map $F$, defined on the set of functions on the conformal boundary at infinity
$\partial_\infty M$ times functions on $M$ times symmetric two-tensor
fields  {on $M$}, mapping to functions on $M$ times symmetric two-tensor fields  {on $M$},
which to $(\hat{U},V,g)$ associates
$$
\left(\begin{array}{c}V(\nabla^*\nabla V+nV+\langle
\Omega(V,g,{\mathring V},{\mathring g}),dV\rangle)+\frac{n-2}{n-1}  |d U|^2\\
\Ric(g)+ng-V^{-1}\nabla_i\nabla_jV+\div^*\Omega(V,g,{\mathring
V},{\mathring g})+ V^{-2}( d U\otimes d U-\frac{1}{n-1}|d U|^2
g)\end{array}\right)\,.
$$

\begin{proposition}\label{Flisse}
Let $\mathring{\griem}={\mathring V}^2d\varphi^2+{\mathring g}$ be an asymptotically
hyperbolic static Einstein metric on $\mbbS^1\times M$, $k\in\N$,
$\alpha\in(0,1)$. The map ${\mathcal   F}$   defined as
$$
\begin{array}{ccc}
 C^{k+2,\alpha}(\partial M)\times C_{1}^{k+2,\alpha}(M)\times
 C_{2}^{k+2,\alpha}(M,{\mathcal
 S}_2)&\longrightarrow&C_{0}^{k,\alpha}(M)\times
 C_{ {2} }^{k,\alpha}(M,{\mathcal  S}_2)
\\
 (\hat{U},W,h)&\longmapsto&F(\hat{U},{\mathring
 V}+W,{\mathring g}+h)
\end{array}
$$
is  smooth  in a neighborhood of zero.
\end{proposition}

\begin{proof}
{Let $\delta $ be as in \eq{16XI16.23}.
We have,} by direct estimations,
$$
V^{-2} (dU\otimes dU-\frac1{n-1} |d U|^2 g)\in  {C_{2 + 2 \delta}^{k,\alpha}(M,{\mathcal S}_2)\subset}
C_{2}^{k,\alpha}(M,{\mathcal S}_2)\,.
$$
{The remaining arguments of the proof of~\cite[Proposition~5.2]{ChDelayStationary} apply with trivial modifications:}
The function ${\mathring V}\in C_{-1}^{k+2,\alpha}(M)$ is strictly
positive, so the same is true for ${\mathring V}+W$ if $W$ is
sufficiently small in $C_{1}^{k+2,\alpha}(M)\subset
C_{-1}^{k+2,\alpha}(M) $. Similarly, the symmetric two-tensor
field ${\mathring g}+h\in C_{0}^{k+2,\alpha}(M,{\mathcal S}_2) $
is positive definite when $h$ is small in
$C_{2}^{k+2,\alpha}(M,{\mathcal S}_2)\subset
C_{0}^{k+2,\alpha}(M,{\mathcal S}_2)$. The map
$(\hat{U},V,g)\mapsto U$ is smooth.
The fact that the remaining terms in $F(\hat{U},{\mathring
V}+W,{\mathring g}+h)$ are in the space claimed, and that the map is
smooth, follows  from standard  calculations (see~\cite[proof of Theorem~4.1]{GL} or  \cite[Appendix]{Delay:etude} for {associated} detailed calculations).
\end{proof}

We are ready to  formulate now:

\begin{theorem}\label{imp3d}
Let $\dim M=n\geq3$ and let ${\mathring V}^2d\varphi^2+{\mathring g}$ be a
non-degenerate asymptotically hyperbolic static Einstein metric on
$\mbbS^1\times M$, $k\in\N$, $\alpha\in(0,1)$, $\delta\in(0,1)$ in $n=3$ and  $\delta=1$
if $n>3$. For all
$\hat{U}$ close to zero in $C^{k+2,\alpha}(\partial M)$  there exists a unique solution
$$(U,V,g)=(\hat{U}+u,{\mathring V}+W,{\mathring
g}+h)$$ to \eq{mainequation} with
$$
(u,W,h)\in C_{\delta}^{k+2,\alpha}(M)\times
C_{1}^{k+2,\alpha}(M)\times C_{2}^{k+2,\alpha}(M,{\mathcal
S}_2)\,,
$$ close to zero, satisfying the gauge condition $\Omega=0$. Moreover,
the map $\hat{U}\mapsto (u,W,h)$ is a smooth map of Banach spaces
near zero.
\end{theorem}

\begin{proof}
As already pointed out, the function
$U=U(\hat{U},V,g)$ exists and is
unique when $W$ and $h$ are small. From Proposition~\ref{Flisse} we
know that the map ${\mathcal F}$ is smooth. The linearisation of
${\mathcal F}$ at zero is
$$
D_{(W,h)}{\mathcal   F}(0,0,0)=D_{(V,g)}F(0,{\mathring
V},{\mathring g})=(l,L)\,.
$$
From Corollary~\ref{isotordu}, with $\delta=2$, we obtain that
$D_{(W,h)}{\mathcal F}(0,0,0)$ is an isomorphism. The implicit
function theorem shows that the conclusion of Theorem~\ref{imp3d}
remains valid \emph{for the modified equation~\eq{jaugeequation}}.
Returning to Section~\ref{secjauge}, we see that
$\Omega=\Omega(V,g,{\mathring V},{\mathring g})\in
C_{2}^{k+1,\alpha}(M,T_1)$.
 {Corollary~\ref{cor:isoform} gives   $\Omega=0$,
and we have obtained a solution to \eq{mainequation}.}
\end{proof}

\noindent{\sc Proof of Theorem~\ref{maintheorem}:}
 Existence and uniqueness of a solution with
 \bel{result+}
 V-\mathring
 V=O(\rho)\,,\quad U=\hat U +O(\rho^{\delta}) ,\quad
 g-\mathring g =\Ogr(\rho^2)=\Ogrbar(1)\,,
 \ee
follows from Theorem~\ref{imp3d}.
 {Here and elsewhere we write $u=O(\rho^\sigma)$ for a tensor $u$ if the coordinate components of $u$ in local coordinates near the boundary are $O(\rho^\sigma)$, and $u=\Ogr(\rho^\sigma)$ if the norm  $|u|_{\mathring g}$ of $u$ with respect to the metric $ \mathring{g} $ is $O(\rho^\sigma)$.}
 Standard arguments show that all the fields $V$, $U$ and $g$ are polyhomogeneous (cf.\ e.g.~\cite{CDLS}, compare~\cite[Section~7]{ChDelayStationary}).

To justify \eq{result}, we plug-in a polyhomogeneous asymptotics as in \eq{result+} in the equations. Since  the trace-free part of the Ricci tensor of $\glorentz=-V^2dt^2+g$ decays in $\mathring{\griem}$-norm as $\rho^4$ near $\rho=0$,
the metric $\glorentz=-V^2dt^2+g$ is vacuum up to this order, which implies that the usual Fefferman-Graham expansion holds for $\glorentz$ up to terms $O_{\mathring{\griem}}(\rho^4 \ln \rho)$.
This leads to the following form of the metric near the conformal boundary
\bel{25XI16.31}
\mathring g=\rho^{-2}(d\rho^2+\mathring h)\;,\;\; \mathring h=\breve{h}+O_{\breve{h}}(\rho^2)\;,\;\; \mathring V=\breve{V}\rho^{-1}+O(\rho)
 \,,
\ee
where $\mathring h$ is a family of Riemannian metrics on the boundary depending upon $\rho$, with
$$
 g=\mathring g+ O_{\mathring g}(\rho^2)\;,\;\;V=\mathring V+O(\rho)
  \,.
$$
The equation $\nabla_i(V^{-1}\nabla^i U)=0$ %
%%
%\beal{25XI16.1}
% \lefteqn{
%    \sqrt{\det \mathring h}\;\partial_\rho\left[\rho^{3-n}(1+O(\rho^2))\partial_\rho U\right]
%    }
%    &&
%\\
% &&
%      =
%    -\rho^{3-n}\partial_A\left[\left(\sqrt{\det \mathring h}+O(\rho^2)\right)(\mathring h^{AB}+O^{AB}(\rho^2))\partial_BU\right]
%%     \,.
%      \nonumber
%\eea
%%
%can be written
in local coordinates $(\rho,x^A)$ as in \eq{25XI16.31} takes the form
\begin{eqnarray*}
\breve{V}^{-1}\sqrt{\det \breve{h}}\;\partial_\rho\left[\rho^{3-n}(1+O(\rho^2))\partial_\rho U\right]
\hspace{3cm}\\=
-\rho^{3-n}\partial_A\left[\left(\breve{V}^{-1}\sqrt{\det \breve{h}}+O(\rho^2)\right)(\breve{h}^{AB}+O^{AB}(\rho^2))\partial_BU\right]\\
+\partial_\rho\left(\rho^{5-n}O^A(1)\partial_A U\right)+\partial_A\left(\rho^{5-n}O^A(1)\partial_\rho U\right).
\end{eqnarray*}
Let $U$ be a polyhomogeneous solution with $U=\hat U+O(\rho^\delta)$.
When $n=3$, we can see that no logarithmic term $O(\rho \ln \rho)$ is needed in $U$ so that $U=\hat U+O(\rho)$.
When $n=4$, we find
\bel{25XI16.21}
U=\hat U
 \underbrace{ {-}
\frac12\breve{V}\breve{\nabla}_A(\breve{V}^{-1}\breve{\nabla}^A\hat U)}_{=: U_{\ln {}}}\rho^2\ln\rho+O(\rho^2)
 \,.
\ee
(In dimensions $n\ge 5$ on expects in general a logarithmic term $O(\rho^{n-2} \ln \rho)$ in an asymptotic  expansion of $U$.)
\qed
\begin{remark}\label{Ulncst}
The equality
$$
2\int_{\partial M}\breve{V}^{-1}\hat UU_{\ln{}}=\int_{\partial M}\breve{V}^{-1}|\breve{\nabla}\hat U|^2,
$$
together with the definition of $U_{\ln}$, proves that $U_{\ln{}}=0$ if and only if ${\hat U}$ is constant.  When $\hat U$ is constant the function $U:= \hat U$ solves the equations and satisfies the right-boundary values, so the associated solution is vacuum.

\end{remark}
\appendix

\section{Static Einstein-Maxwell equations in dimensions $n+1\ge 4$}
 \label{s15XI16.1}
%\reject
% \ptcr{see also the file staticEMequations.tex for the original calculations}
%\input{staticEMequations}
% \reject

Consider a Riemannian manifold $(M,g)$, let us denote by $\nabla $ the Levi-Civita derivative operator associated with $g$. Let $V:M\rightarrow \R$,  and set
$$
 ({\mycal  M}=\R\times  M
  \,,
  \,
  \widetilde{ g}=\epsilon
V^2dt^2+g)\,,\quad   \epsilon=\pm1
 \,,
$$
%,
and
$(x^\mu)=(x^0=t,x^i=(x^1,\,.\,.\,.,x^n))$.
We have
\begin{eqnarray}
&
\phantom{xxxx}
\widetilde{\Gamma}^0_{00}=\widetilde{\Gamma}^0_{ij}=\widetilde{\Gamma}^k_{i0}=0,\;\;
\widetilde{\Gamma}^0_{i0}=V^{-1}\partial_iV,\;\;\widetilde{\Gamma}^k_{00}=-\epsilon
V\nabla^kV,\;\;\widetilde{\Gamma}^k_{ij}={\Gamma}^k_{ij}\,,
&
\\
&
\widetilde{R}^l{}_{ijk}={R}^l{}_{ijk},\;\;\widetilde{R}^l{}_{0j0}=-\epsilon
V\nabla_j\nabla^l
V,\;\;\widetilde{R}^0{}_{ij0}=V^{-1}\nabla_j\nabla_iV\,,
&
\\
&
\widetilde{R}^0{}_{ijk}=\widetilde{R}^l{}_{ij0}=\widetilde{R}^l{}_{0jk}
 %=\widetilde{R}^0{}_{0jk}=\widetilde{R}^0{}_{0j0}
 =0\,,
&
%\\
%&
%\widetilde{R}_{mijk}={R}_{mijk},\;\;\widetilde{R}_{0ijk}=0,\;\;\widetilde{R}_{0ij0}=\epsilon
%V\nabla_j\nabla_i V\,,
%&
\\
&\label {7XI16.1}
\widetilde{R}_{ik}={R}_{ik}-V^{-1}\nabla_k\nabla_iV\,,\;\;\widetilde{R}_{0k}=0\,,\;\;\widetilde{R}_{00}=-\epsilon
V\nabla^i\nabla_iV
 \,,
&
\\
& \label{15XI16.2}
\widetilde{R}=R-2V^{-1}\nabla^i\nabla_iV\,.
\end{eqnarray}

{Consider the $(n+1)$-dimensional Einstein  equations
\bel{20XI16.2}
\tilde R_{\alpha\beta} -\frac12 \tilde R \tilde g_{\alpha\beta} +\Lambda\tilde g_{\alpha\beta} = T_{\alpha\beta}
 \,.
\ee
}
Equations~\eq{20XI16.2} with $\tilde g = -V^2 dt^2 +g$ lead to
\beal{7XI16.2}
% &
%  \displaystyle
% \widetilde R_{\alpha}{}^{\alpha} = \frac{2}{n-1} \left(- (n+1)\Lambda + \underbrace{T_{\alpha}{}^\alpha}_{=:\Tr_{\tilde g} T}\right)
% \,,
% &
%\\
 &
  \displaystyle
 \widetilde R_{\alpha\beta} = \frac{2\Lambda - \Tr_{\tilde g} T}{n-1} \widetilde g_{\alpha\beta} + T_{\alpha\beta}
 \,,
 &
\\
 &
 \label{15XI16.21}
  \displaystyle
 R_{ij} = V^{-1} \nabla_i\nabla_j V + \frac{2\Lambda}{n-1} g_{ij} + T_{ij} - \frac {\Tr_{\tilde g} T } {n-1} g_{ij}
 \,,
 &
\\
  &
  \displaystyle
  V \nabla^i \nabla_i V = V^2\left( T_{\alpha \beta} N^\alpha N^\beta + \frac{\Tr_{\tilde g} T - 2 \Lambda} {n-1}
   \right)
  \,,
  &
\eeal{14XI16.1}
where $N^\alpha\partial_\alpha$ is the unit timelike normal to the level sets of $t$.

For a static electric field $F = d (U dt)$, $\partial_t U =0$, {and $T_{\mu\nu}$ as in \eq{11X16.1}} we find
\beal{14XI16.2}
%
% T_{\mu\nu}
%  &= & F_{\mu\alpha} F_{\nu}{}^{\alpha}-\frac{1}{4}\langle F,F\rangle_{\tilde g}\tilde g_{\mu\nu}
% \,,
%\\
 |F|^2 & = &  - 2 V^{-2} |dU|^2
 \,,
\\
  {\tilde g^{\mu\nu} T_{\mu\nu}} & \equiv &   \Tr_{\tilde g} T=
 \frac {(3-n)} 4 |F|^2
 =
 - \frac {(3-n)} 2  V^{-2}|dU|^2
 \,,
\\
 T_{ij}
  &= & -V^{-2} \nabla_i U \nabla_j U +\frac{1}{2} V^{-2}|dU|^2 g_{ij}
 \,,
\\
 g^{ij} T_{ij}
  &= &  \frac{ (n-2)}2 V^{-2}  |dU|^2
% \,,
%\quad
% T_{00}  = \frac 12 |dU|^2
\\
 T_{\hat 0 \hat 0} & :=&
 T_{\mu\nu}N^\mu N^\nu   = \frac 12 V^{-2}|dU|^2
 \,.
  \label{16XI16.12}
\eea
Choosing $\Lambda$ as in \eq{16XI16.11},
Equation~\eq{14XI16.1} gives
\beal{15XI16.1}
 V \Delta V  & = & \frac {n-2}{ (n-1)} |dU|^2
    -\frac{2 \Lambda V^2}{n-1}
   =  \frac {n-2}{ (n-1)} |dU|^2
    + {n V^2}
 \,.
\eeal{14XI16.3}
%
%\ptc{see also VariousRicciDecompositionObsoleteCalculations.tex}

When $\widetilde{ g}=\epsilon
V^2dt^2+g$ is static and
\begin{eqnarray}
%\nonumber
  \displaystyle
  \phantom{xxxxx}
 T=T_{00}(dx^0)^2+T_{ij}dx^idx^j\,, \quad \partial_0T_{00}=\partial_0T_{ij}=0
 \,,
\eeal{Tstatic}
then
\begin{eqnarray}
%\nonumber
  \displaystyle
  \phantom{xxxxx}
\tilde\nabla^{\alpha}T_{\alpha \beta}dx^\beta=
\big(\nabla^iT_{ij}+ V^{-1}\nabla^iVT_{ij}-\epsilon V^{-3}\nabla_jVT_{00}\big)dx^j
 \,.
\eeal{divTstatic}

\bigskip
\noindent{\sc Acknowledgements} The research of PTC was supported in
part by the Austrian Research Fund (FWF), Project  P 24170-N16.

\bibliographystyle{amsplain}

\bibliography{staticEM-minimal}

\def\polhk#1{\setbox0=\hbox{#1}{\ooalign{\hidewidth
  \lower1.5ex\hbox{`}\hidewidth\crcr\unhbox0}}} \def\cprime{$'$}
  \def\cprime{$'$}
\providecommand{\bysame}{\leavevmode\hbox to3em{\hrulefill}\thinspace}
\providecommand{\MR}{\relax\ifhmode\unskip\space\fi MR }
\providecommand{\MRhref}[2]{%
  \href{http://www.ams.org/mathscinet-getitem?mr=#1}{#2}
}
\providecommand{\href}[2]{#2}
\begin{thebibliography}{10}

\bibitem{mand2}
M.T. Anderson, \emph{{Einstein} metrics with prescribed conformal infinity on
  $4$-manifolds}, Geom. Funct. Anal. \textbf{18} (2001), 305--366,
  arXiv:math.DG/0105243. \MR{2421542}

\bibitem{mand1}
\bysame, \emph{Boundary regularity, uniqueness and non-uniqueness for {AH
  Einstein} metrics on $4$-manifolds}, Adv.\ in Math. \textbf{179} (2003),
  205--249, arXiv:math.DG/0104171. \MR{2010802}

\bibitem{ACD}
M.T. Anderson, P.T. Chru\'{s}ciel, and E.~Delay, \emph{Non-trivial, static,
  geodesically complete vacuum space-times with a negative cosmological
  constant}, Jour.\ High Energy Phys. \textbf{10} (2002), 063, 22 pp.,
  arXiv:gr-qc/0211006. \MR{1951922}

\bibitem{ACD2}
\bysame, \emph{Non-trivial, static, geodesically complete space-times with a
  negative cosmological constant. {II}. {$n\geq5$}}, AdS/CFT correspondence:
  Einstein metrics and their conformal boundaries, IRMA Lect. Math. Theor.
  Phys., vol.~8, Eur. Math. Soc., Z\"urich, 2005, arXiv:gr-qc/0401081,
  pp.~165--204. \MR{MR2160871}

\bibitem{AndChDiss}
L.~Andersson and P.T. Chru\'{s}ciel, \emph{On asymptotic behavior of solutions
  of the constraint equations in general relativity with ``hyperboloidal
  boundary conditions''}, Dissert. Math. \textbf{355} (1996), 1--100 (English).
  \MR{MR1405962 (97e:58217)}

\bibitem{besse:einstein}
A.L. Besse, \emph{Einstein manifolds}, Ergebnisse d. Math. 3. Folge, vol.~10,
  Springer, Berlin, 1987.

\bibitem{BBSLMR}
J.L. Bl\'azquez-Salcedo, J.~Kunz, F.~Navarro-L\'erida, and E.~Radu,
  \emph{{Static Einstein-Maxwell Magnetic Solitons and Black Holes in an Odd
  Dimensional AdS Spacetime}}, Entropy \textbf{18} (2016), 438,
  arXiv:1612.03747 [gr-qc].

\bibitem{ChDelayStationary}
P.T. Chru\'{s}ciel and E.~Delay, \emph{Non-singular, vacuum, stationary
  space-times with a negative cosmological constant}, Ann. Henri Poincar\'e
  \textbf{8} (2007), 219--239. \MR{MR2314449}

\bibitem{CDLS}
P.T. Chru\'{s}ciel, E.~Delay, J.M. Lee, and D.N. Skinner, \emph{Boundary
  regularity of conformally compact {E}instein metrics}, Jour.\ Diff.\ Geom.
  \textbf{69} (2005), 111--136, arXiv:math.DG/0401386.

\bibitem{Delay:etude}
E.~Delay, \emph{\'{E}tude locale d'op\'erateurs de courbure sur l'espace
  hyperbolique}, Jour.\ Math.\ Pures Appl. \textbf{78} (1999), 389--430.

\bibitem{ChBActa}
Y.~Four{\`e}s-Bruhat, \emph{Th\'eor\`eme d'existence pour certains syst\`emes
  d'\'equations aux d\'eriv\'ees partielles non lin\'eaires}, Acta Math.
  \textbf{88} (1952), 141--225.

\bibitem{GL}
C.R. Graham and J.M. Lee, \emph{{Einstein} metrics with prescribed conformal
  infinity on the ball}, Adv.\ Math. \textbf{87} (1991), 186--225.

\bibitem{HerdeiroRadu}
C.A.R. Herdeiro and E.~Radu, \emph{{Static black holes with no spatial
  isometries in AdS-electrovacuum}}, Phys. Rev. Lett. \textbf{117} (2016),
  221102, arXiv:1606.02302 [gr-qc].

\bibitem{Heusler:book}
M.~Heusler, \emph{Black hole uniqueness theorems}, Cambridge University Press,
  Cambridge, 1996.

\bibitem{Lee:fredholm}
J.M. Lee, \emph{Fredholm operators and {E}instein metrics on conformally
  compact manifolds}, Mem. Amer. Math. Soc. \textbf{183} (2006), vi+83,
  arXiv:math.DG/0105046. \MR{MR2252687}

\end{thebibliography}
%../references/newbiblio,%
%../references/reffile,%
%../references/bibl,%
%%../references/statio,%
%../references/hip_bib,%
%../references/newbib,%
%../references/PDE,%
%../references/netbiblio}
%%,%stationary}
\end{document}